\documentclass[11pt]{article}
\usepackage{amsmath, amsthm, amscd, amsfonts, amssymb, graphicx, color}
\usepackage[bookmarksnumbered, colorlinks, plainpages]{hyperref}
\usepackage[utf8]{inputenc}
\newtheorem{thm}{Theorem}[section]
\newtheorem{cor}[thm]{Corollary}
\newtheorem{lem}[thm]{Lemma}
\newtheorem{prop}[thm]{Proposition}
\newtheorem{defn}[thm]{Definition}

\numberwithin{equation}{section}

\newcommand{\be}{\begin{equation}}
\newcommand{\ee}{\end{equation}}
\newcommand{\ben}{\begin{enumerate}}
\newcommand{\een}{\end{enumerate}}
\newcommand{\beq}{\begin{eqnarray}}
\newcommand{\eeq}{\end{eqnarray}}
\newcommand{\beqn}{\begin{eqnarray*}}
\newcommand{\eeqn}{\end{eqnarray*}}

\newcommand{\bpf}{\begin{proof}}
\newcommand{\epf}{\end{proof}}
\newcommand{\bl}{\begin{lem}}
\newcommand{\el}{\end{lem}}
\newcommand{\bp}{\begin{prop}}
\newcommand{\ep}{\end{prop}}
\newcommand{\bd}{\begin{defn}}
\newcommand{\ed}{\end{defn}}
\newcommand{\bt}{\begin{thm}}
\newcommand{\et}{\end{thm}}
\title{On complete Yamabe soliton}
\author { M. Yarahmadi and  B. Bidabad\footnote{Corresponding author.}}
\date{}
\begin{document}
\maketitle
\begin{abstract}
In this work, it is shown that a Riemannian complete shrinking Yamabe soliton  has finite fundamental group and its first cohomology group vanishes.
\end{abstract}
%%%%%%%%%%%%%%%%%%%%%%%%%%%%%%%%%%%%%%%%%%%
\section{Introduction}
 Geometric flows are not only applied in physics and mechanics but also has many real world applications. The Yamabe flow was introduced by R.S. Hamilton in order to solve Yamabe's conjecture, stating that any metric is conformally equivalent to a metric with constant scalar curvature, cf., \cite{Ha}. This fact can be used to deform an arbitrary metric into a metric, which determines topology of the underlying manifold and hence innovate numerous progress in the proof of geometric conjectures. Yamabe flow is an evolution equation on a Riemannian manifold $(M,g)$ defined by
$$\frac{\partial g}{\partial t}=-Rg,\qquad g(t=0):=g_0,$$
where $R$ is the scalar curvature. Under Yamabe flow, the conformal class
of a metric does not change and is expected to evolve a manifold toward one with
constant scalar curvature. Yamabe solitons are special solutions of the Yamabe flow and naturally arise as limits of dilations of singularities in the Yamabe flow. Let $(M,g)$ be  a Riemannian manifold, a triple $(M,g,V)$ is said to be a \emph{Yamabe soliton} if $g$ satisfies  the equation
 \begin{align}\label{Eq;RST}
{2Rg+\mathcal{L}_{_{V}}g=2\lambda g,}
\end{align}
where $V$ is a smooth vector field on $M$, $\mathcal{L}_{_{V}}$  the Lie derivative along $V$ and $\lambda$  a real constant. A Yamabe soliton is said to be shrinking, steady or expanding if $\lambda>0$, $\lambda=0$ or $\lambda<0$, respectively. If the vector field $V$ is gradient of a potential function $f$, then $(M,g,V)$ is said to be \emph{gradient} and (\ref{Eq;RST}) takes the familiar form
\begin{equation*}
{Rg+\nabla\nabla f=\lambda g.}
\end{equation*}
 The Yamabe soliton is said to be compact (resp. complete) if $ (M,g) $ is compact (resp. complete). It is well known the scalar curvature of any compact gradient Yamabe soliton is constant, cf., \cite{DS,HSU}.
A complete shrinking gradient Yamabe solitons (resp. Ricci soliton) under suitable scalar
curvature (resp. Ricci tensor) assumptions have finite topological type, cf., \cite{WU} (resp. cf., \cite{Fang}). We note that the Yamabe flow has some similarities to Ricci flow. Moreover, as Ricci solitons are special solutions of Ricci flow, Yamabe solitons are special solutions of Yamabe flow. It is natural to ask whether classical results for Ricci soliton remain valid in the Yamabe soliton case. Lott has shown that the fundamental group of closed manifold $M$ is finite for any gradient shrinking Ricci soliton, cf., \cite{Lott}. As shown by A. Derdzinski, every compact shrinking Ricci soliton has only finitely many conjugacy classes, cf., \cite{Derdz}. Fern\'{a}ndez L\'{o}pez and Garc\'{i}a R\'{i}o have proved that a compact shrinking Ricci soliton has finite fundamental group, cf. \cite{FG}. Moreover, Wylie has shown that a complete shrinking Ricci soliton has finite fundamental group, cf. \cite{Wylie}. Here, inspiring their works we give a positive answer to the above natural questions on Yamabe soliton.

In the present work, it is shown that if a complete Riemannian manifold $(M,g)$ satisfies the following inequality, then its  fundamental group is finite and the first cohomology group vanishes.
\begin{align}\label{Eq;a}
2Rg+\mathcal{L}_{_{V}}{g}\geqslant 2 \lambda g,
\end{align}
where, $\lambda > 0$ and $V=v^{i}(x) \frac{\partial}{\partial x^{i}}$ is a vector field on $M$. Specially the shrinking Yamabe soliton has finite fundamental group and vanishing first cohomology group. Moreover, the sphere bundle of complete shrinking Yamabe soliton has finite fundamental group.

 %##############################################################################################################################################################################
 %##############################################################%##################################################################################################################################################################
%#################################################################################################################################################################
 \section{An estimation for the distance function}
Let $(M,g)$ be a Riemannian manifold. For a $ (m,n)$-tensor field $ \Omega $ define
\begin{align*}
\Vert\Omega\Vert^2:=\Omega_{i_1...i_m}^{j_1...j_n}\Omega_{j_1...j_n}^{i_1...i_m}=g^{j_1l_1}\cdots g^{j_nl_n}g_{i_1k_1}\cdots g_{i_mk_m}\Omega_{l_1...l_n}^{k_1...k_m}\Omega_{j_1...j_n}^{i_1...i_m}.
\end{align*}
 For any point $p\in M$ we define
\begin{align}\label{Hpp}
I_p=\sup\limits_{x\in\mathcal{B}_p(1)}\Vert Ric_x\Vert.
\end{align}
It is clear that $I_p$ is bounded. Let $\gamma:[0,\rho]\longrightarrow M$ be a regular piecewise $C^{\infty}$ curve in $M$. Consider the rectangle
$$\mathsf R=\{(s,t)\vert \ 0\leqslant s \leqslant \rho,\ \ -\epsilon\leqslant t \leqslant \epsilon\}.$$
A piecewise $C^{\infty}$ variation of $\gamma(s)$ is a continuous and piecewise smooth map $\gamma(s,t)$ from $\mathsf R$ into $M$ such that $\gamma(s,0)$ reduces to the given $\gamma(s)$. Their velocity fields give rise, respectively, to the two vector fields $ T:=\gamma_*\frac{\partial}{\partial s}=\frac{\partial\gamma}{\partial s} $ and $ U:=\gamma_*\frac{\partial}{\partial t}=\frac{\partial\gamma}{\partial t}$,  where $T,U\in T_{\gamma}M$. Now, let $\gamma(s)$, $s\in [0,\rho]$, be a geodesic parameterized by the arc length $s$. The second variation of arc length in Riemannian geometry is expressed by
\begin{align}\label{secondVAL}
 L^{\prime\prime}(0)=\int_0^{\rho}\Big[\Vert\nabla_{T}U^{\perp}\Vert^2-g(R(U^{\perp},T)T,U^{\perp})\Big]ds,
\end{align}
where $U^{\perp}=U-g(U,T)T$ is the normal component of $U$.
\begin{lem}\label{lemma}
Let $(M,g)$ be a complete Riemannian manifold, $p$ and $q$ two points in $M$ such that $\rho:=d(p,q)>1$ and $\gamma$ the minimal geodesic from $p$ to $q$ parameterized by the arc length $ s $, then
$$\frac{1}{2}\int_{0}^{\rho}\Vert Ric{_{\gamma(s)}}\Vert\ ds\leqslant \ n-1+I_p+I_q.$$
\end{lem}
\begin{proof}
Let $\{E_i\}_{i=1}^n$ be a parallel orthonormal frame where $E_n:=\gamma^{\prime}(s)$. Since $\gamma$ is minimal geodesic, it has minimal length in its free homotopy class and $L^{\prime\prime}(0)\geqslant 0$. $E_i$ is orthogonal  to $\gamma^{\prime}$, that is, $g(E_i,\gamma^{\prime})=0$, $1\leqslant i \leqslant n-1$, hence $E_i^{\perp}=E_i$. For any real piecewise smooth function of one variable $\phi$ with $\phi(0)=\phi(\rho)=0$ and $U=\phi E_i$, the second variation of arc length (\ref{secondVAL}) yields
$$0\leqslant\int_0^{\rho}\big(\Vert\nabla_{\gamma^{\prime}}(\phi E_i)\Vert^2-g(R(\phi E_i,\gamma^{\prime})\gamma^{\prime},\phi E_i)\big)ds,\qquad \forall 1\leqslant i\leqslant n-1.$$
This implies
\begin{align}\label{SVAL}
0\leqslant\ \sum\limits_{i=1}^{n-1}\int_0^{\rho}\big(\Vert\nabla_{\gamma^{\prime}}(\phi E_i)\Vert^2-g(R(\phi E_i,\gamma^{\prime})\gamma^{\prime},\phi E_i)\big)ds.
\end{align}
Since $E_n=\gamma^{\prime}$, we have $g(R(\phi E_n,\gamma^{\prime})\gamma^{\prime},\phi E_n)=0$ and
\begin{align}\label{GRR}
\sum\limits_{i=1}^{n-1}g(R(\phi E_i,\gamma^{\prime})\gamma^{\prime},\phi E_i)=\sum\limits_{i=1}^{n}g(R(\phi E_i,\gamma^{\prime})\gamma^{\prime},\phi E_i)=\phi^2Ric(\gamma^{\prime},\gamma^{\prime}).
\end{align}
Replacing $\nabla_{\gamma^{\prime}}(\phi E_i)=\frac{d\phi}{ds}E_i$ and (\ref{GRR}) in (\ref{SVAL}), we get the following estimate
\begin{align}\label{VF}
 0\leqslant \ \int_0^{\rho}\big((n-1)(\frac{d\phi}{ds})^2-\phi^2(s)Ric(\gamma^{\prime},\gamma^{\prime})\big)ds.
\end{align}
On the other hand, the Cauchy-Schwarz inequality yields
\begin{align*}
|Ric(\gamma^{\prime},\gamma^{\prime})|=|{\gamma^{\prime}}^j{\gamma^{\prime}}^kRic_{jk}|\leqslant ({\gamma^{\prime}}^j{\gamma^{\prime}}^k{\gamma^{\prime}}_j{\gamma^{\prime}}_k)^{\frac{1}{2}}(Ric_{jk}Ric^{jk})^{\frac{1}{2}}= \Vert Ric(\gamma(s))\Vert.
\end{align*}
Hence we have
\begin{align}\label{CSH}
-\phi^2(s)\Vert Ric(\gamma(s))\Vert\leqslant \phi^2(s)Ric(\gamma^{\prime},\gamma^{\prime})\leqslant \phi^2(s)\Vert Ric(\gamma(s))\Vert
\end{align}
By means of (\ref{VF}) and (\ref{CSH}) we have
\begin{align}\label{SemiVF}
 0\leqslant \ \int_0^{\rho}\big((n-1)(\frac{d\phi}{ds})^2+\phi^2(s)\Vert Ric(\gamma(s))\Vert\big)ds.
\end{align}
 By decomposition (\ref{SemiVF}) we have
\begin{align}\label{tajziye}
0\leqslant &\int_0^1(n-1)(\frac{d\phi}{ds})^2ds+\int_1^{{\rho}-1}(n-1)(\frac{d\phi}{ds})^2ds+\int_{{\rho}-1}^{\rho}(n-1)(\frac{d\phi}{ds})^2ds\nonumber\\&+\int_0^1\phi^2(s)\Vert Ric(\gamma(s))\Vert\ ds+\int_{1}^{{\rho}-1}\phi^2(s)\Vert Ric(\gamma(s))\Vert\ ds\nonumber\\&+\int_{{\rho}-1}^{\rho}\phi^2(s)\Vert Ric(\gamma(s))\Vert\ ds.
\end{align}
Consider the piecewise smooth function $\phi:[0,{\rho}]\longrightarrow [0,1]$, where
$$
\phi(s)=\left\{
\begin{array}{ccc}
s \qquad&  & 0\leqslant s\leqslant 1, \qquad\\
 1 \qquad&  & 1\leqslant s\leqslant {\rho}-1,\\
 {\rho}-s &  & {\rho}-1\leqslant s\leqslant {\rho}.
 \end{array}\right.
$$
By definition of $\phi$, Equation (\ref{tajziye}) reduces to
\begin{align}\label{tajziye1}
\nonumber0\leqslant \ &2(n-1)+\int_0^1\phi^2(s)\Vert Ric(\gamma(s))\Vert\ ds+\int_{1}^{{\rho}-1}\Vert Ric(\gamma(s))\Vert\ ds\\&+\int_{{\rho}-1}^{\rho}\phi^2(s)\Vert Ric(\gamma(s))\Vert\ ds.
\end{align}
Adding $\int_{0}^{{\rho}}\Vert Ric{(\gamma(s))}\Vert ds$ to both sides of (\ref{tajziye1}) yields
\begin{align}\label{tajziye2}
\nonumber\int_{0}^{{\rho}}\Vert Ric{(\gamma(s))}\Vert ds\leqslant &\ 2(n-1)+\int_0^1(1+\phi^2(s))\Vert Ric{(\gamma(s))}\Vert \ ds\\&+\int_{{\rho}-1}^{\rho}(1+\phi^2(s))\Vert Ric{(\gamma(s))}\Vert \ ds.
\end{align}
The minimizing geodesic $\gamma$ yields $d(p,\gamma(s))=s$. This follows $d(p,\gamma(s))\leqslant 1$ for $0\leqslant s\leqslant 1$. Thus $\gamma(s)\in\mathcal{B}_p(1)$ for $0\leqslant s\leqslant 1$ and by (\ref{Hpp}) we have $ \Vert Ric{(\gamma(s))}\Vert\leqslant I_p $, where $0\leqslant s\leqslant 1$. Therefore, $0\leqslant\phi\leqslant 1$ leads
\begin{align}\label{Hp}
\int_0^1(1+\phi^2(s))\Vert Ric{(\gamma(s))}\Vert \ ds\leqslant 2I_p.
\end{align}
Similarly, the minimizing geodesic $\gamma$ yields $d(\gamma(s),q)={\rho}-s$. Hence $d(\gamma(s),q)\leqslant 1$ for ${\rho}-1\leqslant s\leqslant {\rho}$ and consequently $ \Vert Ric{(\gamma(s))}\Vert\leqslant I_q $. This follows that
\begin{align}\label{Hq}
\int_{{\rho}-1}^{\rho}(1+\phi^2(s))\Vert Ric{(\gamma(s))}\Vert\ ds\leqslant 2I_q.
\end{align}
Replacing (\ref{Hp}) and (\ref{Hq}) in (\ref{tajziye2}) we conclude that
$$\int_{0}^{{\rho}}\Vert Ric{(\gamma(s))}\Vert\ ds\leqslant \ 2(n-1)+2I_p+2I_q.$$
As we have claimed.
\end{proof}
%Let $(M,g)$ be a Riemannian manifold and $V=v^{i}(x) \frac{\partial}{\partial x^{i}}$
% a vector field on $M$. Consider the following inequality
%\begin{align}\label{Eq;a}
%2Rg+\mathcal{L}_{_{V}}{g}\geqslant 2 \lambda g,
%\end{align}
%where, $\lambda > 0$.
\begin{thm}\label{theorem1}
Let $(M,g)$ be a complete Riemannian manifold satisfying (\ref{Eq;a}). Then, for any $p,q\in M$
  \begin{align}
  d(p,q)\leqslant \max\big\{1,\frac{1}{\lambda}\Big(2n^{\frac{1}{2}}\big((n-1)+I_p+I_q\big)+\Vert V_p\Vert+\Vert V_q\Vert\Big)\big\}.
  \end{align}
\end{thm}
\begin{proof}
Let $ p,q $ be two points in $ M $ joined by a minimal geodesic $ \gamma:[0,\infty)\longrightarrow M$ parameterized by the arc length $ s $. If $d(p,q)\leqslant 1$, then the assertion follows directly. Suppose that ${\rho}:=d(p,q)> 1$. The Lie derivative of a Riemannian metric tensor is given by
\begin{align*}
(\mathcal{L}_{_{V}}g)(X,Y)=g(\nabla_XV,Y)+g(X,\nabla_YV).
\end{align*}
 Therefore we have along $ \gamma $
\begin{align}\label{Eq;1}
(\mathcal{L}_{_{V}}g)(\gamma^{\prime},\gamma^{\prime})=g(\nabla_{\gamma^{\prime}}V,\gamma^{\prime})+g(\gamma^{\prime},\nabla_{\gamma^{\prime}}V)=2g(\nabla_{\gamma^{\prime}}V,\gamma^{\prime}).
\end{align}
On the other hand, by  metric compatibility in Levi-Civita connection we have along the geodesic $\gamma$
\begin{align}\label{Eq;3}
g(\nabla_{\gamma^{\prime}}V,\gamma^{\prime})=
\nabla_{{\gamma^{\prime}}}g(V,\gamma^{\prime})=
\frac{d}{ds}(g(V,\gamma^{\prime})).
\end{align}
Replacing (\ref{Eq;3}) in (\ref{Eq;1}) we have
\begin{align}\label{Eq;4}
(\mathcal{L}_{_{V}}g)(\gamma^{\prime},\gamma^{\prime})=2\frac{d}{ds}(g(V,\gamma^{\prime})).
\end{align}
By means of (\ref{Eq;a}) and (\ref{Eq;4}) we get
\begin{align*}
2Rg(\gamma^{\prime},\gamma^{\prime})+2\frac{d}{ds}(g(V,\gamma^{\prime}))\geqslant 2\lambda g(\gamma^{\prime},\gamma^{\prime}).
\end{align*}
Since $\gamma$ is parameterized by arc length, $g(\gamma^{\prime},\gamma^{\prime})=1$. This implies
\begin{align*}
R\geqslant\lambda -\frac{d}{ds}(g(V,\gamma^{\prime})).
\end{align*}
 Integrating both sides of the last equation leads to
\begin{align}\label{UUU}
\int_0^{\rho}R\ ds\geqslant \lambda \rho-g(V,\gamma^{\prime}(\rho))+g(V,\gamma^{\prime}(0)).
\end{align}
On the other hand, using of the Cauchy-Schwarz inequality we have
\begin{align*}
|R|=|g^{ij}Ric_{ij}|=|\langle g,Ric\rangle|\leqslant \ \langle g,g\rangle^{\frac{1}{2}}\langle Ric,Ric\rangle^{\frac{1}{2}}=n^{\frac{1}{2}}\Vert Ric\Vert.
\end{align*}
This implies
\begin{align}\label{YYY}
\int_0^{\rho}R\ ds\leqslant \int_0^{\rho}|R| \ ds \leqslant n^{\frac{1}{2}}\int_0^{\rho}\Vert Ric\Vert \ ds.
\end{align}
Using (\ref{UUU}) and (\ref{YYY}) we have
\begin{align}\label{ZZZ}
n^{\frac{1}{2}}\int_0^{\rho}\Vert Ric\Vert\ ds\geqslant \lambda {\rho}-g(V,\gamma^{\prime}({\rho}))+g(V,\gamma^{\prime}(0)).
\end{align}
The Cauchy-Schwarz inequality yields $|g(V,\gamma^{\prime}(0))|\leqslant \Vert V_p\Vert$ and $|g(V,\gamma^{\prime}({\rho}))|\leqslant \Vert V_q\Vert$. Therefore $-\Vert V_p\Vert\leqslant g(V,\gamma^{\prime}(0))\leqslant \Vert V_p\Vert$ and $-\Vert V_q\Vert\leqslant g(V,\gamma^{\prime}({\rho}))\leqslant \Vert V_q\Vert$. Thus we get
\begin{align}
\label{Vert}-g(V,\gamma^{\prime}({\rho}))+g(V,\gamma^{\prime}(0))\geqslant -\Vert V_q\Vert-\Vert V_p\Vert.
\end{align}
Replacing (\ref{Vert}) in (\ref{ZZZ}) we have
\begin{align}\label{UUV}
n^{\frac{1}{2}}\int_0^{\rho}\Vert Ric\Vert\ ds\geqslant \lambda {\rho}-\Vert V_q\Vert-\Vert V_p\Vert.
\end{align}
By means of (\ref{UUV}) and Lemma \ref{lemma} we have
\begin{align*}
2n^{\frac{1}{2}}\big((n-1)+I_p+I_q\big)\geqslant\lambda {\rho}-\Vert V_p\Vert-\Vert V_q\Vert.
\end{align*}
Finally, we get
\begin{align*}
{\rho}=d(p,q)\leqslant\frac{1}{\lambda}\Big(2n^{\frac{1}{2}}\big((n-1)+I_p+I_q\big)+\Vert V_p\Vert+\Vert V_q\Vert\Big).
\end{align*}
This completes the proof.
\end{proof}
%%%%%%%%%%%%%%%%%%%%%%%%%%%%%%%%
\section{The fundamental group of shrinking Yamabe solitons}
Let $M$ be a connected smooth manifold. There exists a simply connected smooth manifold $\tilde M$, called the universal covering manifold of $M$, and a smooth covering map $p:\tilde M\longrightarrow M$ such that it is unique up to a diffeomorphism. A deck transformation on universal covering manifold $\tilde M$ is an isometry $ h:\tilde M\longrightarrow \tilde M$ such that $p\circ h=p$. The group of all deck transformations on a universal covering manifold $\tilde M$ is isomorphic to the fundamental group $\pi_1(M)$ of $M$.
\begin{thm}\label{thm2}
Let $(M,g)$ be a complete Riemannian manifold satisfying (\ref{Eq;a}). Then the fundamental group $\pi_1(M)$ of $M$ is finite and its first cohomology group vanishes, i.e.,  $H^1_{\mathrm{dR}}(M)={0}$.
\end{thm}
\begin{proof}
Let $ \tilde M $ be the universal covering manifold of $ M $ with the smooth covering map $ p:\tilde M\longrightarrow M $.  Pull back of the Riemannian metric $g$ by $ p $ defines a Riemannian metric on $ \tilde M $ denoted by $ \tilde g:=p^*g$. Therefore $(\tilde M,\tilde g)$ and $(M,g)$ are locally isometric. Let $ \tilde V $ denote the lift of $ V $, that is, $ \tilde V=p^*V $. By means of the local isometry $ p:(\tilde M,\tilde g)\longrightarrow (M,g) $ and the inequality (\ref{Eq;a}), we have
\begin{align*}
 p^*(2Rg+\mathcal{L}_{_{V}}  {g})\geqslant 2 p^*(\lambda g).
\end{align*}
By linearity of ${p}^*$ we get
\begin{align}\label{ineqaulity}
2 p^*(Rg)+ p^*\mathcal{L}_{_{V}}  {g}\geqslant 2\lambda  p^* g.
\end{align}
By means of $ \tilde V = p^*V $ and properties of Lie derivative we obtain
\begin{align}\label{Eq;5}
 p^*\mathcal{L}_{_{V}}  {g}=\mathcal{L}_{_{\tilde V}}  {\tilde g}.
\end{align}
On the other hand,  $p$ is a local isometry and we have $ \tilde R= p^*R $, where $ \tilde R$ is the scalar curvature of Riemannian manifold $(\tilde M,\tilde g)$. Hence $p^*(Rg)=(p^*R)(p^*g)=\tilde R\tilde g$. Therefore by replacing (\ref{Eq;5}) in (\ref{ineqaulity})  we get
\begin{align*}
2\tilde R\tilde g+\mathcal{L}_{_{\tilde V}}  {\tilde g}\geqslant 2\lambda \tilde  g.
\end{align*}
Our universal covering $(\tilde M,\tilde g)$ is geodesically complete because $(M,g)$ is so. In fact, let $\tilde{\gamma}(t)$ be any geodesic emanating from some point $\tilde x\in \tilde M$ at $t=0$. Put $\gamma(t):=p(\tilde{\gamma}(t))$. Hence  $\gamma(t)$ is a geodesic because $(\tilde M,\tilde g)$ and $(M,g)$ are locally isometric. By  completeness assumption of geodesics  on $(M,g)$, $\gamma(t)$ is extendible to all $t\in [0,\infty)$. The said local isometry now implies the same for $\tilde{\gamma}(t)$. Hence the universal covering $(\tilde M,\tilde g)$ is geodesically complete.\\
 Let $h$ be a deck transformation on $\tilde M$ and $\tilde x\in\tilde M $. By definition $h:\tilde M\longrightarrow\tilde M$ is an isometry and  the balls $\mathcal{B}_{\tilde x}(1)$ and $\mathcal{B}_{h(\tilde x)}(1)$ are isometric. Therefore (\ref{Hpp}) yields  $I_{\tilde x}=I_{h(\tilde x)}$ and $ \Vert \tilde V_{\tilde x}\Vert=\Vert \tilde V_{h(\tilde x)}\Vert$.  By means of Theorem \ref{theorem1} for the points $\tilde x$ and $h(\tilde x)$ we get
\begin{align*}
  d(\tilde x,h(\tilde x))\leqslant &\max\big\{1,\frac{1}{\lambda}\Big(2n^{\frac{1}{2}}\big(n-1+I_{\tilde x}+I_{h(\tilde x)}\big)+\Vert \tilde V_{\tilde x}\Vert+\Vert \tilde V_{h(\tilde x)}\Vert\Big)\big\}\\&= \max\big\{1,\frac{2}{\lambda}\Big(n^{\frac{1}{2}}\big(n-1+2I_{\tilde x}\big)+\Vert \tilde V_{\tilde x}\Vert\Big)\big\},
  \end{align*}
for any deck transformation $h$. Thus the set $p^{-1}(x)$, where $ x=p(\tilde x)$, is bounded. Using the geodesically completeness and Hopf-Rinow's theorem, the closed and bounded subset $ p^{-1}(x) $ of $\tilde M$ is compact and being discrete is finite. By assumption, $ M $ is connected, so all of its fundamental groups $ \pi_1(M,x) $ are isomorphic, where $x$ denotes the base point. Since $\tilde M $ is a universal covering, $ \pi_1(M,x) $ is bijective with $ p^{-1}(x) $ and therefore $ \pi_1(M) $ is finite. By a well known theorem the first cohomology group $H^1_{\mathrm{dR}}(M)=0$. This completes the proof.
\end{proof}
In special cases we have the following corollaries.
\begin{cor}
Let $(M,g,V)$ be a complete shrinking Yamabe soliton. Then the fundamental group $\pi_1(M)$ of $M$ is finite and therefore $H^1_{\mathrm{dR}}(M)=0$.
\end{cor}
 Let us denote by $SM$ the sphere bundle defined by $SM:=\bigcup\limits_{x\in M}S_xM$ where, $S_xM:=\{v\in T_xM|g(v,v)=1\}$. $SM$ is a subbundle of the tangent bundle $TM$ which has some applications in extension of Riemannian geometry.
\begin{cor}
Let $(M,g,V)$ be a complete shrinking Yamabe soliton. Then the fundamental group $\pi_1(SM)$ of $SM$ is finite and therefore $H^1_{\mathrm{dR}}(SM)=0$.
\end{cor}
\begin{proof}
Let $ \tilde M $ be the universal covering manifold of $ M $ with the smooth covering map $ p:\tilde M\longrightarrow M $. It is well known that the following homotopy sequence of the fibre bundle $(S\tilde M,\tilde\pi,\tilde M,S^{n-1})$ is exact, that is
\begin{align}\label{sequence}
\cdots\longrightarrow\pi_1(S^{n-1})\longrightarrow\pi_1(S\tilde M)\longrightarrow\pi_1(\tilde M)\longrightarrow\cdots,\end{align}
is exact. Since $\tilde M$ is simply connected, $\pi_1(\tilde M)=0$. We know that $\pi_1(S^{n-1})=0$. Thus by (\ref{sequence}) we get $\pi_1(S\tilde M)=0$. One can easily check that $ p_*:S\tilde M\longrightarrow SM $ is a smooth covering map. Therefore $ S\tilde M $ is the universal covering manifold of $ SM $. According to the proof of Theorem \ref{thm2}, $p^{-1}(x)$, $\forall x\in M $, is a finite set and consequently  ${p}_*^{-1}(y)\subseteq\bigcup\limits_{\tilde x\in{p}^{-1}(x)}S_{\tilde x}\tilde M$, $\forall y\in S_xM $,   is compact and being discrete is finite. Thus the fundamental group $\pi_1(SM)$ is finite and therefore $H^1_{\mathrm{dR}}(SM)=0$.
\end{proof}
%#########################################################################################################################################################################
%$$$$$$$$$$$$$$$$$$$$$$$$$$$$$$$$$$$$$$$$$$$$$$$$$$$$$$$$$$$$$$$$$$$$$$$$$$$$$$$$$$$$$$$$$$$$$$$$$$$$$$$$$$$$$$$$$$$$$$$$$$$$$$$$$$$$$$$$$$$$$$$$$$$$$$$$$$$$$$$$$$$$$$$$$

%\vspace{2cm}
\date{ \small Mohamad Yarahmadi and Behroz Bidabad\\
Department of Mathematics and Computer Science\\
Amirkabir University of Technology (Tehran Polytechnic)\\
424 Hafez Ave. 15914 Tehran, Iran.\\
E-mail: m.yarahmadi@aut.ac.ir; bidabad@aut.ac.ir}

\end{document}